\documentclass[leqno,11pt,twoside]{article} 
\usepackage{amsmath, amsthm, amssymb}

\usepackage{geometry}
\usepackage{fancyhdr} 
\usepackage{titlesec}

\titleformat{\section}
  {\normalfont\scshape}{\thesection .}{1em}{}
\titleformat{\subsection}
  {\normalfont\scshape}{\thesubsection .}{1em}{}

\title{An infinitesimal approach to a conjecture of Eisenbud and Harris}
\date{}
\author{\scshape J\"urgen Rathmann}

\makeatletter
\let\runauthor\@author
\let\runtitle\@title
\makeatother

\usepackage{mathrsfs}  
\usepackage[cmtip]{xypic}
\usepackage{enumitem}
\setlist{nolistsep}

\newtheorem{theorem}{Theorem}[section]
\newtheorem{proposition}[theorem]{Proposition}

\newtheorem{corollary}[theorem]{Corollary}
\newtheorem{conjecture}[theorem]{Conjecture}

\newtheorem{question}[theorem]{Question}

\theoremstyle{remark}
\newtheorem{remark}[theorem]{Remark}
\newtheorem{example}[theorem]{Example}

\newcommand\F{\mathscr{F}}
\newcommand\I{\mathscr{I}}
\renewcommand\O{\mathscr{O}}
\renewcommand\P{\mathbf{P}}
\renewcommand\lto{\longrightarrow}

\DeclareMathOperator{\Ker}{Ker}
\DeclareMathOperator{\Img}{Im}
\DeclareMathOperator{\Coker}{Coker}

\DeclareMathOperator{\Hom}{Hom}

\DeclareMathOperator{\rk}{rk}
\DeclareMathOperator{\Cliff}{Cliff}

\begin{document}
\bibliographystyle{plain}
\maketitle

\section{Introduction}

\overfullrule=0pt

Eisenbud and Harris conjectured that an algebraic curve of high genus 
lies on a surface of low degree (which they proved for curves
of very large degree). 
They observed constraints on the Hilbert function of a general hyperplane section 
$\Gamma$, which imply that $\Gamma$ lies on a curve $C$ of low degree. 
We investigate this situation under deformation. Given a set of 
sufficiently many points $\Gamma\subset C$ with $C$ linearly normal,
we show that for every deformation of $\Gamma\subset\P^r$ with constant 
$h_\Gamma(2)$ the curve $C$ deforms along with $\Gamma$.

\bigskip
Let $X$ be a reduced irreducible nondegenerate curve of degree $n$ and genus $g$ in $\P^{r+1}$.
Castelnuovo showed in the late 19th century that the genus is bounded by a quadratic
polynomial $\pi_0(n,r)$ in $n$ with leading term $\frac{n^2}{2r}$. He also showed that curves
of maximal genus must lie on a unique surface $S$ of minimal degree $r$, at least as long as
$n\ge 2r+3$.

Castelnuovo's approach was extended by Eisenbud and Harris in 1982 who defined functions
$\pi_\alpha(n,r)=\frac{n^2}{2(r+\alpha)}+O(n)$
($1\le \alpha\le r$) and $n_0(r)$ and showed that any curve $X$ with degree
$n\ge n_0$ and genus $g>\pi_\alpha(n,r)$ must lie on a surface of degree
$\le r-1+\alpha$. {\cite[3.22]{Harris}}.
Here $n_0(r)= 2^{r+2}$ for $r\ge 7$ (slightly worse for $r<7$).

They further conjectured that the restriction on $n$ could be lowered:

\begin{conjecture}[Eisenbud-Harris]
Let $X\subset\P^{r+1}$ be a nondegenerate reduced irreducible curve of genus $g$ and
degree $n\ge 2r+3+2\alpha$. If $g>\pi_\alpha(n,r)$\textup{,} then $X$ lies on a surface of degree
at most $r-1+\alpha$.
\end{conjecture}

They proved the conjecture for $\alpha=1$ \cite[3.15]{Harris}. The next case $\alpha=2$ ($r\ge 7$)
was established by Petrakiev in 2008 \cite{Petra}.

The proofs proceed by a careful analysis of the Hilbert function $h_\Gamma(l)$ of a general hyperplane 
section $\Gamma=X\cap H$. Here $h_\Gamma(l)$ is the rank of the $l$-th graded piece of the
homogeneous coordinate ring of $\Gamma$. 

Castelnuovo obtains his bound from showing that $h_\Gamma(l)\ge \min(n,lr+1)$.  
On the other hand, if $n\ge 2r+3$ and $h_\Gamma(2)=2r+1$, then 
the quadrics containing $\Gamma$ intersect in a unique rational normal curve of 
degree $r$ in $H$. The surface $S$ of minimal degree is obtained as the intersection of 
all quadrics which contain $X$.

Eisenbud and Harris' proof for $\alpha=1$ uses that $h_\Gamma(2)=2r+2$
for curves with $n\ge 2r+5$ and $g>\pi_1(n,r)$. They show that the quadrics through $\Gamma$
intersect in an elliptic normal curve.

These approaches suggest the following supporting conjecture which (except for part (ii)) 
has been noted and studied by Reid \cite{Reid}, and later by Petrakiev \cite{Petra}:

\begin{conjecture}\label{con11}
Let $\Gamma \subset \P^r$ be a general hyperplane
section of a nondegenerate reduced irreducible curve of degree $n$.  
Assume 
that\ \ $\Gamma$ imposes $2r+1+\alpha$ conditions on quadrics for some 
$0 \le \alpha \le r-2$.  If one of the following holds
\begin{itemize} 
\item[\textup{(i)}] $\alpha\le r-3$ and $n \ge 2r+3+2\alpha$\textup{,} 
\item[\textup{(ii)}] $\alpha=r-2$ and $n \ge 4r+1=2r+5+2\alpha$\textup{,}
\end{itemize}
then $\Gamma$ lies on a reduced irreducible curve $C$ of degree at most $r+\alpha$. $C$ is
a component of the intersection of the quadrics containing $\Gamma$.
\end{conjecture}

For fixed $r$, the numerical restrictions on $n$ and $\alpha$ cannot be 
improved \cite{Reid},\cite{EGH}.

In addition to the cases $\alpha=0, 1$ mentioned above, this conjecture is known 
for $\alpha=2$, $r\ge 5$ and $\alpha=3$, $r\ge 7$ \cite[prop.\ 4.3 and 4.4]{Petra}.
For arbitrary $\alpha\le r-3$, it is known, if $\Gamma$ lies on a
two-dimensional rational normal scroll \cite{Fano}, \cite{Reid}.

\bigskip
In this note, we try to collect further evidence by investigating
the following infinitesimal version of (\ref{con11}):

\begin{question}
Suppose $\Gamma$ is a set of $n$ points in $\P^r$ satisfying both the 
assumptions and the conclusion of \textup{(\ref{con11})}.  Is every deformation of\ \ $\Gamma$
induced by a deformation of $C$\textup{?}
\end{question}

Our setup deals with the simplest case, namely that of linearly normal smooth
curves. 

To describe our results, let $0 \le \alpha \le r-2$ and consider the open subset 
of the Hilbert scheme of curves of genus $g=\alpha$ and degree $d=g+r$ corresponding 
to smooth irreducible curves embedded by a complete linear system.  Using the 
universal family, we can construct a variety $H(n,\alpha)$ parametrizing subschemes 
of such curves consisting of pairwise distinct points.  $H(n,\alpha)$ is smooth 
and irreducible.

Further, let $P(n)$ be the open subset of the Hilbert scheme
of length $n$-subschemes of $\P^r$ corresponding to subschemes $\Gamma$
consisting of pairwise distinct points, and consider the canonical
restriction map 
\begin{equation*}
\xymatrix{H^0\O_{\P^r}(2) \ar[r]^{M} & H^0\O_{\Gamma}(2)}
\end{equation*}
on $P(n)$.
\begin{equation*}
P(n,\alpha)=\{(x_1,\ldots,x_n) \in P(n): \rk(M)\le 2r+\alpha+1\}
\end{equation*}is a 
determinantal variety.

The embedding of $C$ in $\P^r$ yields a map $F(n,\alpha)\colon H(n,\alpha) \lto P(n,\alpha)$.

\begin{theorem}\label{thm13}
If\ \ $0 \le \alpha \le r-3$ and $n \ge 2r+3+2\alpha$
or $\alpha=r-2$ and $n \ge 4r+1$\textup{,} then $F(n,\alpha)$ maps $H(n,\alpha)$ isomorphically
on a dense open subset of a component of $P(n,\alpha)$.
\end{theorem}

\begin{theorem}\label{thm14}
If\ \ $0 \le \alpha \le r-2$ and $n \ge 2r+\alpha+5$\textup{,} 
and $C$ is a fixed curve\textup{,} then 
$F(n,\alpha)$ is locally an isomorphism in a general subscheme $\Gamma$\! of $C$.
\end{theorem}

\begin{corollary}
If\ \ $0 \le \alpha \le r-2$ and $n \ge 2r+\alpha+5$\textup{,} then $P(n,\alpha)$ 
has an irreducible component whose general point lies in the image of $F(n,\alpha)$.
\end{corollary}

The bounds in (\ref{thm13}) correspond to those predicted by (\ref{con11}). 
They are sharp, see (\ref{ex34}).

Our proofs proceed by analyzing the corresponding map on the 
Zariski tangent spaces.
Each statement is equivalent to the surjectivity of the canonical composite map
\begin{equation*}
H^0\I_C(2)\otimes H^0B\lto 
H^0(N^\vee_{C\vert\P}(2))\otimes H^0B\lto
H^0\big(N^\vee_{C\vert\P}(2)\otimes B\big).
\end{equation*}
for $B=\omega_C\otimes\O_C(\Gamma-2)$. 

Writing
\begin{equation*}
P_L=\Ker\big((H^0\I_C(2))\otimes\O_C\lto N^\vee_{C\vert\P}(2)\big),
\end{equation*}
this question reduces to showing that $H^1(P_L\otimes B)=0$.

As $P_L$ is a quotient of 
$(pr_1)_*\big(pr_2^*L\otimes pr_3^*L\otimes\O(-\Delta_{1,2}-\Delta_{1,3}-\Delta_{2,3})\big)$
($pr_i\colon C\times C\times C\lto C$ the canonical projections), we are
led to study conditions on line bundles $L_1$, $L_2$ and $L_3$ which guarantee the
vanishing of
\begin{equation*}
H^1\big(pr_1^*L_1\otimes pr_2^*L_2\otimes pr_3^*L_3\otimes\O(-\Delta_{1,2}-\Delta_{1,3}-\Delta_{2,3})\big).
\end{equation*}

The latter question can be resolved by a suitable filtration of a direct image, 
similarly to the approach in \cite{Laz}.

Interestingly, our proof does not require the points of $\Gamma$ to be in linearly general 
position. If the curve $C$ has degree $2g+2$, it cannot have a trisecant line, but may have 
a foursecant plane. A $5$-secant plane is not possible, but there could be a
$5$-secant $3$-plane; for hyperelliptic $C$, there could even be a $6$-secant-$3$-plane.

The proof even works for an arbitrary divisor on $C$ without assuming that
its support consists of pairwise distinct points.

In the last section, we briefly comment on our assumption of linear normality of $C$ 
and give two applications of our methods, on the question of generation of ideal
sheaves by quadrics and on the cohomology of the square of the ideal sheaf.

\bigskip
These results were obtained in 1990--91 at UCLA, but remained unpublished
when I did not continue my academic career.
As I recently found that the techniques can be applied to study syzygies of curves, 
I decided to publish the results, after updating my notes with more recent developments.

I am grateful to Rob Lazarsfeld for many helpful discussions and encouragement.  
I am also indebted to Aaron Bertram who first asked the questions 
answered by (\ref{thm32}) and (\ref{prop42}) (they were communicated
to me by Lazarsfeld) which turned out to be pivotal for this paper.
Furthermore, I profited from discussions with Mark Green and Christoph Rippel.

\section*{Notations and conventions}

We work over an algebraically closed field $K$ of arbitrary characteristic.

Given a curve $C$, 
we will write $pr_1$, $pr_2$ for the projections on $C \times C$,
and $\Delta$ for the diagonal.  On $C \times C$, we define sheaves 
associated with a line bundle $L$  as follows:  
$M_L=(pr_1)_*\big(pr_2^*L\otimes \O_{C\times C}(-\Delta)\big)$,
$R_L=(pr_1)_*\big(pr_2^*L\otimes \O_{C\times C}(-2\Delta)\big)$.

If $L$ is very ample and $C \lto \P^r$ is the embedding by $H^0L$, 
then $M_L=\Omega^1_\P(1)\vert_C$, $R_L=L\otimes N^{\vee}_{C/\P}$.

$pr_i$ also denotes the projections from 
$C^3=C \times C \times C$ onto the $i$-th
component, and $pr_{1,2}$ denotes the projection on the first two components.

$\Delta_{i,j}$ is the diagonal $\{(x_1,x_2,x_3)|x_i=x_j\}\subset C\times C\times C$.

For a line bundle $L$, let
$M_{2,L}=(pr_{1,2})_*\big(pr_3^*L\otimes\O_{C\times C\times C}(-\Delta_{1,3}-\Delta_{2,3})\big)$.
If $L$ is very ample, then $M_{2,L}$ is a vector bundle on $C\times C$.

For very ample $L$ so that the ideal of $C$ in the embedding by $H^0L$ is generated 
by quadrics, we write $P_L=\Ker\big(H^0\I_C(2)\otimes\O_C\lto N^\vee_{C/\P}(2)\big)$.

\section{Deforming sets of points}

Let $C$ be a curve of genus $g$, embedded in $\P^r$ by a complete linear system of 
degree $d=g+r$ for some $r \ge g +2$. Then $\deg(C)\ge 2g+2$ and $C$ 
satisfies $(N_1)$, i.e., the homogeneous ideal of $C$ is generated by quadrics.
Note that $C$ imposes $h^0\O_C(2)=2d+1-g= 2r+1+g$ conditions on quadrics.

In the situation of (\ref{thm13}), 
$C$ is uniquely determined as the intersection of the quadrics containing 
$\Gamma$, since $\Gamma$ consists of $n \ge 2r+3+2g=2d+3>2d$ points.  

In the situation of (\ref{thm14}), 
the same holds for general $\Gamma$, because the residual divisor for a quadric 
through $n$ general points on $C$ consists of a section of a line bundle of degree
$2d-n\le g-5<g$, and the general such line bundle has only the zero section.

As $F(n,g)$ is injective at $\Gamma$ under our assumptions, it suffices to
show that the induced map on Zariski tangent spaces is an isomorphism.
We will use the following local criterion:

\begin{proposition}\label{prop21}
Let $C$ be a curve of genus $g=\alpha$,
embedded into $\P^r$ by a non-special complete linear system.  
Assume that $C$ satisfies $(N_1)$.
Let $\Gamma=\{x_1,\ldots,x_n\} \subset C$
with $H^0\O_C(2-\Gamma)=0$ and set $B=\omega_C\otimes \O_C(\Gamma-2)$.
The following conditions are equivalent\textup{:}
\begin{itemize}
\item[\textup{(i)}] $F(n,g)$ induces a bijective map on Zariski tangent spaces
in $\Gamma$.
\item[\textup{(ii)}] The multiplication map $H^0\big(\I_C(2)\big)\otimes H^0B \lto H^0 \big(N^{\vee}(2)
\otimes B\big)$ is surjective.
\end{itemize}
\end{proposition}

\begin{proof}
The Zariski tangent space of $H(n,g)$ has a filtration with two components,
namely $H^0N_{\Gamma/ C}$ representing deformations of $\Gamma$ in $C$,
and $H^0N_{C/\P}$ representing deformations of $C$ in $\P^r$.

$P(n,g)$ is a determinantal variety, and its Zariski tangent space at $\Gamma$
is obtained by pullback from the generic determinantal variety 
$M_k=\{A \in K^{m \times n}: \rk(A) \le k\}$ (see \cite[\hbox{${\rm I}\!{\rm I}$.2}]{ACGH}).
The Zariski tangent space at $A \in M_k-M_{k-1}$ is 
$T_A(M_k)=\{B \in K^{m \times n}:B\cdot \Ker(A) \subset \Img(A) \}$. 

In our case, $T_{\Gamma}\big(P(n)\big)=H^0N_{\Gamma/\P}$,
$\Ker(M)=H^0\I_{\Gamma}(2)$, $\Img(M)=\Img\big(H^0\O_{\P^r}(2)\lto H^0\O_{\Gamma}(2)\big)$.
Therefore
\begin{equation*}
T_{\Gamma}\big(P(n,g)\big)=\Ker\big(H^0N_{\Gamma/\P} \lto
\Hom(H^0\I_{\Gamma}(2),H^1\I_{\Gamma}(2))\big).
\end{equation*}

Now consider the following diagram:
\begin{equation*}
\xymatrix{
& 0 \ar[d] \\
& H^0N_{\Gamma/C} \ar[d] \\
& H^0N_{\Gamma/\P} \ar[r] \ar[d] & \Hom\big(H^0\I_{\Gamma}(2),H^1\I_{\Gamma}(2)\big) \\
H^0N_{C/\P} \ar[r] & H^0(N_{C/\P}\vert_\Gamma) \ar[r] \ar[d]
& H^1\big(N_{C/\P}\otimes\O_C(-\Gamma)\big) \ar[r] & H^1N_{C/\P} \\
& 0
}
\end{equation*}
The first map in the bottom line is injective: $H^0\big(N_{C/\P}\otimes
\O_C(-\Gamma)\big)=0$, because $\Gamma$ determines $C$, hence
no non-trivial deformation of $C$ can fix $\Gamma$.

$H^1N_{C/\P}=0$, because $C$ is embedded by a nonspecial linear
system, hence $C$ is unobstructed.

$H^0N_{\Gamma/C}$ classifies deformations of $\Gamma$ inside $C$. 
This vector space lies in the Zariski tangent space of $P(n,g)$ at
$\Gamma$, and we obtain a map
\begin{equation*}
H^0(N_{C/\P}\vert_{\Gamma}) \lto \Hom\big(H^0\I_{\Gamma}(2),H^1\I_{\Gamma}(2)\big).
\end{equation*}

$H^0N_{C/\P}$ classifies deformations of $C$ in $\P^r$. As 
the deformed curve lies on the same number of quadrics, 
this vector space also lies in the Zariski tangent space of $P(n,g)$
at $\Gamma$.

We obtain: 
$C$ deforms with $\Gamma$ if and only if the map
\begin{equation*}
f:H^1\big(N_{C/\P} \otimes \O_C(-\Gamma)\big) \lto
\Hom\big(H^0\I_{\Gamma}(2),H^1\I_{\Gamma}(2)\big)
\end{equation*}
 is injective.

We identify $H^0\I_{\Gamma}(2) = H^0\I_C(2)$ and
$H^1\I_{\Gamma}(2)=H^1\I_{\Gamma,C}(2)=H^1\O_C(2-\Gamma)$,
and dualize $f$:
\begin{equation*}
\xymatrix{
\Hom\big(H^1\O_C(2-\Gamma),H^0\I_C(2)\big) \ar[r] \ar[d]^\cong & H^1(N_{C/\P} \otimes
\O_C(-\Gamma))^{\vee} \\
H^1\O_C(2-\Gamma)^{\vee} \otimes H^0\I_C(2)
}
\end{equation*}
Applying Serre duality, we find a map
\begin{equation*}
f^\vee\colon H^0\big(\omega_C\otimes\O_C(\Gamma-2)\big)\otimes H^0\I_C(2) \lto
H^0\big(N_{C/\P}^{\vee}\otimes \omega_C\otimes\O_C(\Gamma)\big)
\end{equation*}
which is the composition
\begin{align*}
H^0\big(\omega_C\otimes\O_C(\Gamma-2)\big)\otimes H^0\I_C(2) & \lto
H^0\big(\omega_C\otimes\O_C(\Gamma-2)\big)\otimes H^0N_{C/\P}^{\vee}(2) \\
& \lto H^0\big(N_{C/\P}^{\vee}\otimes \omega_C\otimes\O_C(\Gamma)\big).
\end{align*}

$C$ deforms with $\Gamma$ if and only if $f^\vee$ is surjective.
\end{proof}

\begin{proposition}\label{prop22}
Let $L$ be a very ample line bundle on a curve $C$ and assume that the
ideal sheaf of $C$ is generated by quadrics. There is an exact sequence
\begin{multline*}
0 \lto \wedge^2 M_L \lto (pr_1)_*\big(pr_2^*L\otimes pr_3^*L\otimes
\O_{C\times C\times C}(-\Delta_{1,2}-\Delta_{1,3}-\Delta_{2,3})\big) \\ 
\lto \big(H^0\I_C(2)\big)\otimes\O_C \lto N_{C/\P}^{\vee}(2) \lto 0.
\end{multline*}
\end{proposition}

\begin{proof} 
On $C \times C \times C$, we have the following diagram with exact
rows and columns
\begin{equation}\label{diag1}
\begin{aligned}
\xymatrix{
& 0 \ar[d] & 0 \ar[d] & 0 \ar[d] \\
0 \ar[r] & \O(-\Delta_{1,3} -\Delta_{2,3}) \ar[r] \ar[d] & 
\O(-\Delta_{2,3}) \ar[r] \ar[d] &
\O_{\Delta_{1,3}}(-\Delta) \ar[r] \ar[d] & 0\\
0 \ar[r] & \O(-\Delta_{1,3}) \ar[r] \ar[d] & \O \ar[r] \ar[d] &
\O_{\Delta_{1,3}} \ar[r] \ar[d] & 0 \\
0 \ar[r] & \O_{\Delta_{2,3}}(-\Delta) \ar[r] \ar[d] & 
\O_{\Delta_{2,3}} \ar[r] \ar[d] & 
\O_{\Delta} \ar[r] \ar[d] & 0\\
& 0 & 0 & 0 
}
\end{aligned}
\end{equation}
Tensoring with $pr_3^*L$ and applying $(pr_{1,2})_*$, we obtain a diagram on $C\times C$
\begin{equation}\label{diag2}
\begin{aligned}
\xymatrix{
& 0 \ar[d] & 0 \ar[d] & 0\ar[d] \\
0 \ar[r] & M_{2,L} \ar[r] \ar[d] & pr_2^*M_L \ar[r] \ar[d] & pr_1^*L\otimes \O(-\Delta) 
\ar[r] \ar[d] & 0\\
0 \ar[r] & pr_1^*M_L \ar[r] \ar[d] & H^0L\otimes\O \ar[r] \ar[d] & pr_1^*L \ar[r] \ar[d] & 0\\
0 \ar[r] & pr_2^*L\otimes\O(-\Delta) \ar[r] \ar[d] & pr_2^*L \ar[r] \ar[d]
& L\otimes\O_{\Delta} \ar[r] \ar[d] & 0\\
& 0 & 0 & 0
}
\end{aligned}
\end{equation}

Applying $(pr_1)_*$ to the second exterior power of the top horizontal sequence, we find that
$(pr_1)_*(\wedge^2M_{2,L})$ injects into 
$(pr_1)_*(\wedge^2pr_2^*M_L)=(H^0(\wedge^2M_L))\otimes\O_C$. However, $\wedge^2M_L$
has no global sections, hence $(pr_1)_*(\wedge^2M_{2,L})=0$.

Considering next $(pr_1)_*$ of the left vertical sequence, we find a monomorphism
$\wedge^2M_L\lto (pr_1)_*\big(M_{2,L}\otimes pr_2^*L\otimes\O(-\Delta)\big)$.

We now tensor diagram (\ref{diag2}) with $pr_2^*L\otimes\O(-\Delta)$, apply $(pr_1)_*$
and divide the four direct image sheaves in the top left square by
the image of $\wedge^2M_L$.

The middle horizontal row of the resulting diagram appears as the bottom row 
of the following diagram
\begin{equation*}
\xymatrix{
& \wedge^2 M_L \ar@{=}[r] \ar@{^{(}->}[d] & \wedge^2M_L \ar@{^{(}->}[d] \\
0 \ar[r]  & M_L\otimes M_L \ar[r] \ar@{>>}[d] & M_L\otimes H^0L \ar[r] \ar@{>>}[d] & M_L\otimes L
\ar[r] \ar@{=}[d] & 0\\
0 \ar[r] & S^2M_L \ar[r] & \F \ar[r] & M_L\otimes L \ar[r] & 0
}
\end{equation*}
We find that $\F$ (defined here as a push-down) is part of the filtration of $S^2H^0L\otimes\O_C$, 
induced by taking second symmetric powers of
\begin{equation*}
0 \lto M_L \lto H^0L\otimes\O_C \lto L \lto 0.
\end{equation*}
For the full filtration, see the middle row of the next commutative diagram 
\begin{equation}
\begin{aligned}\label{diag2a}
\xymatrix{
& H^0\I_C(2)\otimes \O_C \ar@{^{(}->}[d] \ar@{=}[r]  & H^0\I_C(2)\otimes \O_C \ar@{^{(}->}[d] \\
0 \ar[r] & \F \ar[r] \ar@{>>}[d] & S^2H^0L\otimes\O_C \ar[r] \ar@{>>}[d] & L^2 \ar[r] \ar@{=}[d] & 0\\
0 \ar[r] & M_{L^2} \ar[r] & H^0L^2\otimes\O_C \ar[r] & L^2 \ar[r] & 0
}
\end{aligned}
\end{equation}
The left vertical sequence thus is the middle vertical sequence of the desired result diagram, and
we obtain
\begin{equation}\label{diag3}
\begin{aligned}
\xymatrix{
& 0 \ar[d] & 0 \ar[d] & 0 \ar[d] \\
0 \ar[r] & P_L \ar[r] \ar[d] & H^0\I_C(2)\otimes\O_C \ar[r] \ar[d] & 
R_L\otimes L \ar[r] \ar[d] & 0\\ 
0 \ar[r] & S^2M_L \ar[r] \ar[d] & \F \ar[r] \ar[d] & M_L\otimes L \ar[r] \ar[d] & 0\\
0 \ar[r] & R_{L^2} \ar[r] \ar[d] & M_{L^2} \ar[r] \ar[d] & \omega_C\otimes L^2 \ar[r] \ar[d] & 0\\
& 0 & 0 & 0 
}
\end{aligned}
\end{equation}

The sequence from our proposition arises from the top row, taking into account the
exact sequence
\begin{equation*}
0\lto \wedge^2M_L\lto (pr_1)_*\big(pr_2^*L\otimes pr_3^*L\otimes
\O_{C\times C\times C}(-\Delta_{1,2}-\Delta_{1,3}-\Delta_{2,3})\big)\lto P_L\lto 0
\end{equation*}
used to construct $P_L$.

Finally, we need to show that the map $\big(H^0\I_C(2)\big)\otimes\O_C \lto N^{\vee}_{C/\P}(2)$
is induced by the canonical map $\I_C\lto \I_C/\I_C^2\cong N^\vee_{C/\P}$.
This can be done by a direct calculation, or as follows:

Note that diagram (\ref{diag2}) above can alternatively be obtained
as pullback of a similar diagram that lives on the blow-up of $\P^r\times\P^r$ in the diagonal.

Now given an arbitrary quadric $Q$ containing $C$, pulling the diagram back to the blow-up 
of $Q\times Q$ in its diagonal, and taking the appropriate direct images, 
we obtain the top map of the following diagram
\begin{equation}\label{diag4}
\begin{aligned}
\xymatrix{
\big(H^0\I_Q(2)\big)\otimes\O_Q \ar[r]^-{\sim} \ar[d] & N^{\vee}_{Q/{\bf P}}(2) = \O_Q \ar[d] \\
\big(H^0\I_C(2)\big)\otimes\O_C \ar[r] & N^{\vee}_{C/{\bf P}}(2).
}
\end{aligned}
\end{equation}
The commutativity of diagram (\ref{diag4}) now provides the desired identification of the map.
\end{proof}

\section{Conclusion of the proof}

\begin{theorem}\label{thm31}
$H^0\big(\I_C(2)\big) \otimes H^0B \lto H^0 \big(N^{\vee}_C(2)
\otimes B\big)$ is surjective\textup{,} if one of the following conditions holds\textup{:}
\begin{itemize}
\item[\textup{(i)}] $g=0$\textup{,} $d=\deg(C)\ge 1$\textup{,} $\deg B\ge 1$\textup{;}
\item[\textup{(ii)}] $g=1$\textup{,} $d\ge 4$\textup{,} $\deg B\ge 3$\textup{;}
\item[\textup{(iii)}] $g\ge 2$\textup{,} $d = 2g+2$\textup{,} $\deg B\ge 2g+3$\textup{;}
\item[\textup{(iv)}] $g\ge 2$\textup{,} $d \ge 2g+3$\textup{,} $\deg B\ge 2g+1$.
\end{itemize}
\end{theorem}

\begin{proof}[Proof of \textup{(\ref{thm31})}]  
By (\ref{prop22}) it suffices to show that
$H^1(C,(pr_1)_*\F)=0$ for $\F=pr_1^*B\otimes pr_2^*L\otimes
pr_3^*L\otimes \O(-\Delta_{1,2}-\Delta_{1,3}-\Delta_{2,3})$ where $L=\O_C(1)$.
If we knew that the higher direct images
$R^1(pr_1)_*\F$ and $R^2(pr_1)_*\F$ both vanish, then Leray's
spectral sequence implies that $H^1(pr_1)_*\F=H^1\F$, and the latter
vanishes by (\ref{thm32}) resp.\ (\ref{rmk33}) below.

To compute the higher direct images, we note that
$R^1(pr_{1,2})_*\F=pr_1^*B\otimes pr_2^*L\otimes\O(-\Delta)\otimes 
R^1(pr_{1,2})_*\big(pr_3^*L \otimes\O(-\Delta_{1,3}-\Delta_{2,3})\big)=0$,
because $L$ is non-special and very ample, hence $R^2(pr_1)_*\F=0$,
and $R^1(pr_1)_*\F=B\otimes R^1(pr_1)_*(M_{2,L} \otimes pr_2^*L\otimes
\O(-\Delta))$.

Now consider the left vertical sequence of diagram (\ref{diag2}) in the proof
of (\ref{prop22}).
After tensoring with $pr_2^*L\otimes\O(-\Delta)$ and applying $(pr_1)_*$, 
we find that the stalks of 
$R^1(pr_1)_*\big(M_{2,L}\otimes pr_2^*L\otimes\O(-\Delta)\big)$ sit in the exact sequence
\begin{multline*}
M_{L,x} \otimes H^0L(-x) \overset{f}\lto H^0 L^2(-2x)  
\lto  R^1(pr_1)_*\big(M_{2,L}
\otimes pr_2^*L\otimes\O(-\Delta)\big)_x \\
\lto M_{L,x} \otimes H^1L(-x).
\end{multline*}
$M_{L,x}$ can be canonically identified with $H^0L(-x)$, and $f$
with the multiplication map $H^0L(-x)\otimes H^0L(-x) \lto H^0L^2(-2x)$.
The assumption on $d$ implies that $f$ is surjective and that the right term vanishes 
for every $x \in C$, therefore $R^1(pr_1)_*\F=0$.
\end{proof}

\begin{proof}[Proof of \textup{(\ref{thm13})}]  
The line bundle $B=\omega_C\otimes\O_C(\Gamma-2)$
has degree $2g-2+n-2d=n-2r-2$ which is at least $2g+1$ (for $g\le r-3$)
resp.\ $2g+3$ (for $g=r-2$), therefore the theorem follows from
(\ref{prop21}) and (\ref{thm31}).
\end{proof}

\begin{proof}[Proof of \textup{(\ref{thm14})}]  
$B=\omega_C\otimes\O_C(\Gamma-2)$ is a general 
line bundle of degree $2g-2+n-2d\ge g+3$, and, going back to the proofs of 
(\ref{thm13}) and (\ref{thm31}), it suffices to show that $H^1\big(pr_1^*B\otimes pr_2^*\O_C(1)
\otimes pr_3^*\O_C(1)\otimes \O(-\Delta_{1,2}-\Delta_{1,3}-\Delta_{2,3})\big)=0$.
But that result will be shown as (\ref{eq1}) in the proof of (\ref{thm32}) below.
\end{proof}

\begin{theorem}\label{thm32}
Let $C$ be a curve of genus $g$\textup{,} let $L_1$\textup{,} $L_2$ and 
$L_3$ be line bundles on $C$ with
$\deg L_1\ge 2g+1$\textup{,} $\deg L_2\ge 2g+2$\textup{,} $\deg L_3\ge 2g+3$.
Then 
\begin{equation*}
H^1 \big(pr_1^*L_1\otimes pr_2^*L_2\otimes pr_3^*L_3\otimes
\O_{C \times C \times C}(-\Delta_{1,2}-\Delta_{1,3}-\Delta_{2,3})\big)=0.
\end{equation*}
\end{theorem}

\begin{proof} 
Choose a line bundle $L'_3$ of degree $g+3$ with the following
properties:
\begin{itemize}
\item[\textup{(i)}] $L'_3$ is very ample and non-special,
\item[\textup{(ii)}] $H^1(L_1(p)\otimes {L'_3}^{-1})=0$ for general $p \in C$,
\item[\textup{(iii)}] $H^1( \wedge^2 M_{L'_3} \otimes L_1) =0$,
\item[\textup{(iv)}] $H^1(L_2\otimes {L'_3}^{-1})=0$.
\end{itemize}

Each of them separately holds on a Zariski-open subset of the Picard variety of
line bundles of degree $g+3$,
as long as $\deg\big(L_1(p)\otimes {L'_3}^{-1}\big)\ge g-1$ and
$\deg\big(L_2\otimes {L'_3}^{-1}\big)\ge g-1$.

Since $\deg(L_3\otimes {L'_3}^{-1}) \ge g$, we have $H^0(L_3\otimes {L'_3}^{-1}) \not= 0$.
If the divisor of the corresponding section consists of pairwise distinct points, we have a short exact sequence
\begin{multline*}
0 \lto pr_3^*L'_3 \otimes \O_{C\times C\times C}(-\Delta_{1,3}-\Delta_{2,3}) \lto 
pr_3^*L_3 \otimes \O_{C\times C\times C}(-\Delta_{1,3}-\Delta_{2,3}) \\
\lto \oplus_i\O_{x_3=p_i}(-p_i,-p_i) \lto 0;
\end{multline*}
in the general case, we will still have a short exact sequence starting with the two terms above.
However, the term on the right may not split as a direct sum, but instead have a filtration
with graded pieces $\O_{x_3=p_i}(-p_i,-p_i)$.

Now (\ref{thm32}) will follow from
\begin{gather}\label{eq1}
H^1\big( pr_1^*L_1\otimes pr_2^*L_2\otimes pr_3^*L'_3\otimes
\O_{C\times C\times C}(-\Delta_{1,2}-\Delta_{1,3}-\Delta_{2,3})\big)=0,\\
\label{eq2}
H^1\big( pr_1^*L_1(-p_i) \otimes pr_2^*L_2(-p_i)\otimes\O_{C\times C}(-\Delta)\big)=0.
\end{gather}

It is well-known that (\ref{eq2}) holds, if $\deg L_1(-p_i) \ge 2g$ and 
$\deg L_2(-p_i) \ge 2g+1$; i.e., if $\deg L_1 \ge 2g+1$ and $\deg L_2 \ge 2g+2$.

Since $L'_3$ is very ample and non-special, 
$H^1L'_3(-x_1-x_2) = 0$  for all $(x_1,x_2) \in C \times C$; hence
$R^1(pr_{1,2})_*\big(pr_3^*L'_3\otimes\O_{C\times C\times C}(-\Delta_{1,3}-\Delta_{2,3})\big) = 0$. 
Denoting $M_{2,L'_3}=(pr_{1,2})_*\big(pr_3^*L'_3 \otimes 
\O_{C\times C\times C}(-\Delta_{1,3}-\Delta_{2,3})\big)$,
Leray's spectral sequence reduces (\ref{eq1}) to
\begin{equation}\label{eq3}
H^1 \big(M_{2,L'_3} \otimes pr_1^*L_1 \otimes pr_2^*L_2\otimes
\O_{C\times C}(-\Delta)\big) = 0.
\end{equation}

In the next step, apply $(pr_{1,2})_*$ to
\begin{multline*}
0 \lto pr_3^*L'_3\otimes\O_{C\times C\times C}(-\Delta_{1,3}-\Delta_{2,3}) \lto
pr_3^*L'_3\otimes\O_{C\times C\times C}(-\Delta_{1,3}) \\
\lto pr_3^*L'_3 \otimes\O_{\Delta_{2,3}}(-\Delta_{1,3}) \lto 0
\end{multline*}
to obtain a sequence 
\begin{equation*}
0 \lto M_{2,L'_3} \lto pr_1^*M_{L'_3} \lto pr_2^*L'_3
\otimes\O_{C\times C}(-\Delta) \lto 0.
\end{equation*}
Its second exterior power is
\begin{equation*}
0 \lto \wedge^2 M_{2,L'_3} \lto \wedge^2 pr_1^*M_{L'_3} \lto M_{2,L'_3}
\otimes pr_2^*L'_3 \otimes\O_{C\times C}(-\Delta) \lto 0.
\end{equation*}
Since $M_{2,L'_3}$ is a vector bundle of rank $H^0L'_3(-x-y)=2$, 
$\wedge^2 M_{2,L'_3}$ is a line bundle, we calculate
\begin{equation*}
\wedge^2 M_{2,L'_3}=pr_1^*{L'_3}^{-1} \otimes pr_2^*{L'_3}^{-1}
\otimes \O_{C\times C}(\Delta),
\end{equation*}
and (\ref{eq3}) will now follow from
\begin{gather}\label{eq4}
H^1 \big(pr_1^* (\wedge^2 M_{L'_3} \otimes L_1)
\otimes pr_2^* (L_2\otimes {L'_3}^{-1})\big) = 0;\\
\label{eq5}
H^2 \big(pr_1^*(L_1\otimes {L'_3}^{-1})\otimes pr_2^*(L_2\otimes
{L'_3}^{-2})\otimes \O_{C\times C}(\Delta) \big)= 0.  
\end{gather}

Now (iii) and (iv) together imply (\ref{eq4}), while
(\ref{eq5}) follows from Leray's spectral sequence, since (ii) above implies that
$R^1(pr_2)_*\big(pr_1^*(L_1\otimes {L'_3}^{-1})\otimes\O_{C\times C}(\Delta)\big)$ 
is supported in a (possibly empty) set of points.
\end{proof}

\begin{remark}\label{rmk33}
For curves of genus $g\le 1$, some of the 
arguments in the above proof can be improved:  

For elliptic curves, 
any line bundle $L'_3$ of degree 3 is very ample and gives a plane embedding.
Hence the vanishing result of (\ref{thm32}) holds for line bundles $L_1$, $L_2$, $L_3$,
as soon as $\deg L_1\ge 3$, $\deg L_2\ge 4$ and $\deg L_3\ge 4$. Therefore
(\ref{thm13}) even holds for $\alpha=1$, $r=3$, 
in line with Eisenbud and Harris' result.

For rational curves, we have the vanishing for $\deg L_1\ge 1$, 
$\deg L_2\ge 1$ and $\deg L_3\ge 1$, hence (\ref{thm13})
is also valid for $\alpha=0$, $r=2$ in line with Castelnuovo's result.
\end{remark}
\smallskip

The bounds in (\ref{thm31}) and (\ref{thm32}) are sharp, as can be seen from the following

\begin{example}\label{ex34}
Let $C$ be a hyperelliptic curve of genus $g\ge 2$.
\begin{itemize}
\item[\textup{1.}] ($\deg L_1=\deg L_2=2g+1$, arbitrary $L_3$)  Embed $C$ by a 
line bundle $L$ of degree $2g+1$.  Then 
$(H^0\I_C(2))\otimes\O_C \lto N^{\vee}_C(2)$ is not surjective.
\item[\textup{2.}]  ($\deg L_1=\deg L_2=\deg L_3=2g+2$)  
Embed $C$ by a line bundle $L$ of degree $2g+2$.  
Then $H^0\I_C(2) \otimes H^0L \lto H^0 N^{\vee}_C(3)$ is not surjective.
\end{itemize}
\end{example}

\begin{proof}
1. Since $C$ is hyperelliptic, its image in $\P^{g+1}$ lies on a rational normal 
scroll $S$ of degree $g$ \cite{EisHar}. Any quadric not containing $S$ will intersect
$S$ in a divisor of degree $2g$, hence cannot contain $C$.

2. According to the exact sequence
\begin{equation*}
0\lto P_L\lto \big(H^0\I_C(2)\big)\otimes\O_C\lto N_C^\vee(2)\lto 0,
\end{equation*}
it suffices to show that $H^1(P_L\otimes L)\not=0$. As $H^1(\wedge^2M_L\otimes L)=0$, the exact sequence
\begin{equation*}
0\lto \wedge^2M_L\lto (pr_1)_*
\big(pr_2^*L\otimes pr_3^*L\otimes\O_{C\times C\times C}(-\Delta_{1,2}-\Delta_{1,3}-\Delta_{2,3})\big) \lto P_L\lto 0
\end{equation*}
on $C$ and Leray's spectral sequence show that 
\begin{equation*}
H^1(P_L\otimes L)\cong H^1\big(pr_1^*L\otimes pr_2^*L\otimes pr_3^*L\otimes
\O_{C\times C\times C}(-\Delta_{1,2}-\Delta_{1,3}-\Delta_{2,3})\big).
\end{equation*}
Filtering $L$ as in the proof of (\ref{thm32}), we find that this cohomology group
surjects onto $H^1(C\times C,pr_1^*L(-p)\otimes pr_2^*L(-p)\otimes\O_{C\times C}(-\Delta))$
which is non-trivial for hyperelliptic $C$ \cite[1.4.2]{Laz}.
\end{proof}

\section{On linear normality, and two applications}

\subsection{About the assumption of linear normality}\label{ssec41}

Recall the setup in which we are working: 

$X\subset\P^{r+1}$ is a nondegenerate reduced irreducible curve of 
degree $n$ and genus $g>\pi_\alpha(n,r)$ for some $\alpha\le r$. 
$C\subset\P^r$ is a curve of degree $\le r-1+\alpha$ containing a 
general hyperplane section $\Gamma$ of $X$.

Our results in sections 2 and 3 apply to linearly normal curves $C$
whose ideal sheaf is generated by quadrics. 

We would like to point out:
\begin{itemize}
\item[(i)] Even if $X$ is linearly normal, $C$ may not have this property;
e.g., 
let $C_0$ be a curve of genus $g\ge 1$, let $D$ be a line of degree $2g+k$,
let $S=C_0\times\P^1$, embedded into $\P^{2g+2k+1}$ by the
line bundle $pr_1^*D\otimes\O_{\P^1}(1)$ as a surface of degree $4g+2k$, 
and let $X$ be the intersection of $S$ with a hypersurface of large degree;
here $\alpha=2g+1$ and $h^1\I_C(1)=h^2\O_S=g$.
\item[(ii)] The intersection of the quadrics containing $\Gamma$
may have dimension $2$; let $k=1$ and let $C_0$ be hyperelliptic in the
previous example; the ideal sheaves of $S$ and its
hyperplane section $C$ are generated by quadrics and cubics.
\end{itemize}

\subsection{A criterion for a curve to be locally an intersection of quadrics}

Let $L$ be a very ample line bundle of degree $d\ge 2g+2-2h^1(L)-\Cliff(C)$
defining an embedding $C\subset \P(H^0(L))=\P^r$. Green and Lazarsfeld 
have shown that $C$ is scheme-theoretically cut out by quadrics unless it 
has a tri-secant line. Their proof is outlined in \cite[2.4.2]{Laz}. 

We would like to offer an alternative proof, starting with

\begin{proposition}
Let $L$ be a very ample line bundle on a curve $C$ which embeds $C$ as a quadratically 
normal curve\textup{,} $x\in C$. Then the following conditions are equivalent\textup{:}
\begin{itemize}
\item[\textup{(i)}] $C$ is scheme-theoretically cut out by quadrics at $x$.
\item[\textup{(ii)}] The line bundle $L(-x)$ is very ample \textup{(}i.e.\textup{,} $C$ has no 
trisecant line through $x$\textup{)} and embeds $C$ as a quadratically normal curve.
\end{itemize}
\end{proposition}

\begin{proof}
Consider the diagram (\ref{diag3}) in the proof of (\ref{thm32}). 
Note first that the two bottom rows of the diagram exist for an arbitrary variety and are exact.
Quadratic normality implies the surjectivity of the vertical map in the middle
(compare diagram (\ref{diag2a})), and we identified
the kernels of the vertical maps, hence the snake lemma yields an exact sequence
\begin{equation*}
H^0\I_C(2) \otimes \O_C \lto N^{\vee}_{C/\P}(2) \lto
\Coker(S^2M_L \lto R_{L^2}) \lto 0.
\end{equation*}
Since the map $S^2M_{L,x} \lto R_{L^2,x}$ can be identified with $S^2H^0L(-x) 
\lto H^0L^2(-2x)$, this implies the equivalence of the two conditions.
\end{proof}

Green and Lazarsfeld's result is now an immediate 
consequence of their earlier result that the embedding by a very ample line bundle 
of degree $d\ge 2g+1-2h^1(L)-\Cliff(C)$ is quadratically normal \cite{GL}.

\subsection{Cohomology of the square of the ideal sheaf}

Our results permit to answer a question, posed independently by A.\ Bertram 
and J.\ Wahl:

\begin{proposition}\label{prop42}
Let $C\subset\P^r$ be a curve of genus $g$\textup{,} embedded
by a complete linear system of degree $d\ge 2g+3$.
Then $H^1\I_C^2(3)=0$. 
\end{proposition}

\begin{proof} It suffices to show that
$H^0\I_C(3) \lto H^0N^{\vee}_{C/\P}(3)$ is surjective.  But by (\ref{thm31}),
the map $H^0\I_C(2)\otimes H^0\O_{\P^r}(1) \lto H^0N^{\vee}_{C/\P}(3)$
is surjective, and it factors through $H^0\I_C(3)$. 
\end{proof}

(\ref{ex34}) shows that the vanishing fails for hyperelliptic curves embedded by
a line bundle of degree $2g+2$, hence our result is the best possible without
further assumptions on $C$.

Wahl is particularly interested in the canonical embedding \cite{Wahl}. He shows that 
the vanishing of $H^1\I_C^2(3)$ holds for the general canonical curve of genus $g\ge 3$,
but fails for curves with $\Cliff(C)\le 2$.

The picture for canonical curves is almost completed by recent work of Arbarello, Bruno 
and Sernesi \cite{ABS} who show that $H^1\I_C^2(3)$ vanishes for all curves of genus $g\ge 8$ 
and $\Cliff(C)\ge 3$.

A natural generalization of (\ref{prop42}) in the spirit of \cite{GL} 
would be the following

\begin{conjecture}
Let $C\subset\P^r$ be a curve of genus $g$\textup{,} embedded by a complete linear system
$\vert L\vert$ of degree $d$\textup{,} and suppose that $h^1L\le 1$.
If $d\ge 2g+3-2h^1(L)-\Cliff(C)$\textup{,} then $H^1\I_C^2(3)=0$.
\end{conjecture}

\bigskip
\textsc{Flurstra\ss e 49, 82110 Germering, Germany}

\medskip
\textit{E-mail address}: {\tt Juergen\_Rathmann@yahoo.com}

\end{document}